\documentclass{article} 
\usepackage{arxiv}

\usepackage[english]{babel} 
\usepackage{amssymb}
\usepackage{amsmath}
\usepackage{amsthm}
\usepackage{mathtools}

\usepackage{url}

\usepackage{abstract}

\newtheorem{definition}{Definition}
\newtheorem{lemma}{Lemma}
\newtheorem{theorem}{Theorem}
\newtheorem{corollary}{Corollary}
\newtheorem{proposition}{Proposition}
\newtheorem{remark}{Remark}
\newtheorem{assumption}{Assumption}

\usepackage{xcolor}
\usepackage{pgfplots}
\usepackage{bbm}
\pgfplotsset{compat=1.16}

\title{Revisiting semidefinite programming approaches to options pricing: complexity and computational perspectives}

\author{
	Didier Henrion
	\thanks{CNRS-LAAS, Toulouse, France, henrion@laas.fr, 
	Faculty of Electrical Engineering, Czech Technical University in Prague}
	\and
	\textbf{Felix Kirschner}
	\thanks{Tilburg University,	Tilburg, the Netherlands, f.c.kirschner@tilburguniversity.edu   } 
	\and 
	\textbf{Etienne de Klerk}
	\thanks{Tilburg University,	Tilburg, the Netherlands, e.deklerk@tilburguniversity.edu }
	\and
	\textbf{Milan Korda}
	\thanks{CNRS-LAAS, Toulouse, France, korda@laas.fr, Faculty of Electrical Engineering, Czech Technical University in Prague }
	\and
	\textbf{Jean-Bernard Lasserre}
	\thanks{CNRS-LAAS, Toulouse, France, lasserre@laas.fr}
	\and
	\textbf{Victor Magron}
	\thanks{CNRS-LAAS, Toulouse, France, magron@laas.fr}}

\begin{document}
	\maketitle
	
	\begin{abstract}
			In this paper we consider the problem of finding bounds on the prices of options depending on multiple assets without assuming any underlying model on the price dynamics, but only the absence of arbitrage opportunities. We formulate this as a generalized moment problem and utilize the well-known Moment-Sum-of-Squares (SOS) hierarchy of Lasserre to obtain bounds on the range of the possible prices. A complementary approach (also due to Lasserre) is employed for comparison. We present several numerical examples to demonstrate the viability of our approach. The framework we consider makes it possible to incorporate different kinds of observable data, such as moment information, as well as observable prices of options on the assets of interest. 
	\end{abstract}

	\maketitle

	\section{Introduction}\label{sec:intro}
	Derivative securities have become an integral part in financial economics and constitute attractive instruments for a wide variety of parties. Such products may be used to hedge portfolios, ensure financial planning security in supply chains and for investment purposes. The value of a derivative security relies on the value of one or multiple assets, called \emph{underlyings}, like stocks, currencies, commodities or similar. The most commonly used derivative securities are futures, forwards, swaps or options. A central question of financial economics is at what price to sell such products. Important in this respect is to ensure the price put on the security creates no possibility of arbitrage, i.e., there must not be a risk-free possibility to make money. Two main approaches to finding bounds on such prices are used throughout the literature. The first one assumes the prices of the underlying assets follow a stochastic differential equation (SDE) and tools from the theory of SDEs are used to solve the problem of finding a price. The most famous model in this regard is the Black and Scholes model, which provides closed formula solutions to many problems. However, this has the drawback that the assumed model is highly susceptible to model misspecifications and to parameter estimation errors. The other approach, which is the one we will follow, has no underlying model or assumptions on the price dynamics, but solely assumes the non-existence of arbitrage. It is based on the idea of using observable data like prices of other options on the same asset or prices of correlation-based derivatives and then using semidefinite optimization techniques to obtain solutions.

	In this paper we will focus on the problem of deriving bounds on the price of \emph{European call options}.  A European call option is a contract that gives the owner the right, but no obligation, to buy an underlying asset at fixed price, referred to as \emph{strike} (or \emph{strike price}) at a predetermined date in the future, called \emph{maturity}. Since the owner is not obliged to exercise the option, it has nonnegative value. For example, consider a European call option with strike $K$ on a stock, whose price at time $t$ is given by $S_t$. If at maturity $T$ the price $S_T$ of the stock is greater than the strike price $K$, a rational owner will exercise the option an make a profit of $S_T-K$. If, however, the price of the asset is less than the strike, the owner will not exercise the option (since they could buy the stock cheaper at the stock market) and therefore not make a profit. Thus, the \emph{payoff function} of the option is given by $\max\{ S_T -K, 0 \}$. There are many different types of options and we will introduce the ones that will be relevant in this paper. A \emph{rainbow option} is an option on multiple underlyings $S_T^{(1)}, \dots, S_T^{(n)}$ that pays on the level of one option. For example a call on max with payoff function $\max(0,\max\{S_t^{(1)}, \dots, S_t^{(n)}\}-K)$. This is equivalent to a \emph{lookback option} on one asset if $S_t^{(i)}$ is the price of the same asset at $n$ points in time. A \emph{basket option} also depends on multiple assets and pays on the level of more than one. For example, it could be a weighted linear combination of the prices of the assets at maturity with payoff function $\max(0,\sum_{i=1}^n \alpha_i S_t^{(i)}-K)$, where $\alpha_i \ge 0$. Examples for options of this type are index options or currency basket options. Because markets are incomplete in general, it is not possible to compute exact prices of options. However, one can compute bounds, such that, if the price of the option lies within the given range, it is consistent with the given information and does not create the possibility of arbitrage.
	
	\subsection{Prior work}
	The problem of computing bounds on option prices without assuming a specific price dynamic of the underyling assets has been studied since the 1970s beginning with the poineering work of Merton \cite{merton}. Cox and Ross \cite{cox} and Harrison and Kreps \cite{harrison} show that the assumption of no arbitrage possibilities is equivalent to the existence of a probability measure under which the option prices become Martingales. Boyle and Lin \cite{boyle} extended prior contributions of Lo \cite{lo} considering the problem of deriving upper bounds on basket options on multiple assets given the means and the covariance matrix of the underlying assets by constructing a semidefinite program. In \cite{bertsimas}, Bertsimas and Popescu considered a more general setting assuming observable options prices as well as moment information of the underlying distribution of the assets like means and variances are available. Using semidefinite programming techniques they solve the univariate case and give closed form solutions to some cases. For the multivariate case, i.e., options depending on multiple assets they prove that the problem is NP-hard in general and present a relaxation to the problem by enlarging the set of possible values the assets can attain. They follow up by identifying the cases in which their relaxation can be solved efficiently, which is the case if the objective and the constraint functions are quadratic or linear over $d$ disjoint polyhedra $D_1, \dots, D_d$ which form a partition of $\mathbb{R}^n$, where $n$ is the number of assets considered. Davis and Hobson \cite{davis} study the structure of the underlying problem and give sufficient and necessary conditions for the existence of measures specified in \cite{cox}, \cite{harrison}. In a series of papers (see \cite{laurence}, \cite{laurence2}, \cite{Hobson2005}) Hobson, Laurence, and Wang  consider the case of multivariate basket options and give sharp upper and lower bounds when the constraints consist of observable vanilla options prices. They do not employ semidefinite programming techniques, but approach the problem by constructing primal and dual solutions with a zero duality gap. Primbs \cite{Primbs2006} constructs dynamic replicating portfolios using semidefinite programming to get upper and lower bound on option prices, using knowledge of piecewise polynomial data. In his dissertation \cite{daspremont}, d'Aspremont computes bounds for basket options by constructing static replicating portfolios assuming knowledge on prices of different basket options with the same maturity. Li et al. \cite{han} extend the work of Bertsimas and Popescu using sum-of-squares (SOS) relaxations to obtain a hierarchy of bounds on option prices. Another approach was taken by Peña and Zuluaga \cite{Zuluaga2005}. They used tools from conic programming to reformulate the considered problem and prove strong duality in many cases. To give approximate solutions to the problem they propose to use increasingly tight outer approximations of the cone of interest. For certain sets $K$ they provide explicit outer approximation sequences for the cone of measures supported on $K$, and use these to compute upper bounds for option prices. 
	
	\subsection{Contribution of this paper}
		Our work builds on the work of Bertsimas and Popescu \cite{bertsimas}. We analyze and computationally explore cases which they simply determined to be NP-hard. We consider a model similar to the one treated by Li et al. \cite{han}, which in itself is a generalization of the problem Boyle and Lin \cite{boyle} considered. While the authors in \cite{han} focus on a dual approach using inner (i.e., SOS) approximations of the cone of positive polynomials, our main interest lies in a primal method relying on an outer approximation of the moment cone. In contrast to Li et al. we give a rigorous argument as to why we consider compact underlying sets whenever we do so. To complement our primal method of outer approximation we analyze an inner approximation of the moment cone as well. Our inner approximation does not rely on any compactness assumption. In special cases we give explicit bounds on the support of the optimal solution of the treated problem. Our method of outer approximations takes the same approach as Peña and Zuluaga in \cite{Zuluaga2005}. Our analysis contributes additional insights into when optimal solutions exist and the proposed hierarchies converge. Several numerical examples are provided to illustrate the effectiveness of our methods.  
		
		\subsection{Outline of the paper}
	
	We first introduce some notation and give insights to the generalized moment problem and the Moment-SOS hierarchy as tools to approximate such problems in section \ref{sec:preliminaries}. 
	
	
	Following that, in section \ref{sec:outer} we present the problem we intend to study in this paper, which is finding bounds on the prices of options depending on multiple assets without assuming any underlying stochastic processes of the assets prices. This can be modeled as a generalized moment problem over a non-compact set.
	
	
	Also in section \ref{sec:outer}, we prove the existence of an optimal solution of the problem formulation we proposed. Equipped with this knowledge we continue in section \ref{sec:boundsupport} to apply the core variety procedure to a special case to demonstrate how to obtain a bound on the support of the optimal solution. 
	
	
	Section \ref{sec:examplesOuter} contains a few examples of numerical computations for problems with real world data as well as some explanation of the implementation techniques. In section \ref{sec:noncompact} we apply a relaxation technique for the non-compact generalized moment problem to our setting and conclude the section with a numerical example to show its effectiveness.
	

	\section{Preliminaries}\label{sec:preliminaries}
	\subsection{Notation}
	Let $\mathbb{N} = \{0, 1, 2, \dots \}$ be the set of nonnegative integers. We will denote by $\mathbb{R}[\textbf{x}] = \mathbb{R}[x_1, \dots, x_n]$ the ring of real multivariate polynomials in $n$ variables $\textbf{x} = (x_1, \dots, x_n)$. For a vector $\alpha \in \mathbb{N}^n$ with $\alpha = (\alpha_1, \dots, \alpha_n)$ we define $\textbf{x}^\alpha = x_1^{\alpha_1}\cdot \ldots \cdot x_n^{\alpha_n}$. The degree of a monomial $\textbf{x}^\alpha$ is defined as $\vert \alpha \vert = \alpha_1 + \dots+ \alpha_n$ and the degree $\textrm{deg}(p)$ of a polynomial $p \in \mathbb{R}[\textbf{x}]$ is the largest degree of all monomials it consists of. Further, let $\mathbb{N}^n_r = \{ \alpha \in \mathbb{N}^n : \alpha_1 + \dots + \alpha_n \le r \}$. The notation $[m]$ for $m \in \mathbb{N}\setminus \{ 0 \}$ denotes the set $\{ 1, \dots, m \}$. For $r \in \mathbb{N}$ we define $[\textbf{x}]_r$ to be a monomial basis vector of $\mathbb{R}[\textbf{x}]_r$, i.e.,
	\[
	[\textbf{x}]_r^T = (1, x_1, \dots, x_n, x_1^2, \dots, x_n^r).
	\]
	A polynomial $p \in \mathbb{R}[\textbf{x}]$ is called a \emph{sum of squares}, abbreviated SOS, if it can be written as a sum of squared polynomials, i.e. if 
	\[
	p = \sum_{i=1}^m p_i(\textbf{x})^2,
	\]
	for $p_i \in \mathbb{R}[\textbf{x}]$. Given $\omega \in \mathbb{R}$ and $\tilde x \in \mathbb{R}^n$, a weighted Dirac delta measure $\omega \delta_{\tilde{x}}$ with weight $\omega$ is a (atomic) measure with all its mass concentrated on $\tilde{x}$. 
	If $\omega = 1$ then it is a probability measure. For a set $K \subset \mathbb{R}^n$ we denote by $\mathcal{M}(K)_+$ the set of positive  finite Borel measures supported on $K$. 
	By $\mathbbm{1}_{A}(\textbf{x})$ we denote the indicator function of the set $A$. We denote by $\mathbb{S}^n$ the set of $n \times n$ symmetric matrices, by $\mathbb{S}^n_+$ the set of positive semidefinite matrices, and we write $X \succeq 0$ for $X \in \mathbb{S}^n_+$.
	
	\subsection{GMP and Moment-SOS-hierarchy}
	
	Let $K \subset \mathbb{R}^n$. The generalized moment problem (GMP) is an optimization problem of the following form
	
	\begin{equation}\label{gengmp}
	\begin{aligned}
	\inf_{\mu \in \mathcal{M}(K)_+} & \int_K f_0(\textbf{x}) \mathrm{d}\mu(\textbf{x}) \\
	\text{ s.t. } & \int_K f_i(\textbf{x}) \mathrm{d}\mu(\textbf{x}) = a_i \,, \text{ for } i \in [m_1] \\
	& \int_K g_i(\textbf{x})\mathrm{d}\mu(\textbf{x}) \le b_i\,, \text{ for } i \in [m_2] \,,
	\end{aligned}
	\end{equation}
	for $m_1, m_2 \in \mathbb{N}$, $a_i, b_j \in \mathbb{R}$  and $f_i, g_j \in \mathbb{R}[\textbf{x}]$ for all $i \in [m_1],j \in [m_2]$. Since the vector space $\mathcal{M}(K)$ of signed Borel measure is infinite dimensional, this problem is an infinite dimensional conic linear optimization problem, whose duality theory is well understood, see, e.g., \cite{shapiro}. It is straightforward to extend this problem to a more general case where one optimizes over multiple measures supported on different sets and we refer to \cite{Tacchi2021} for an analysis of the more general case. We will use this framework to study the problem of pricing options in this paper.

	
	Many  NP-hard problems can be modeled via the GMP, see e.g.\  \cite{deKlerk}.
	Thus, solving the GMP in full generality is rather hopeless. However, it is possible to construct relaxation hierarchies, whose optimal values serve as bounds on the optimal value and often one can prove they converge to the optimal value. One such hierarchy is the Moment-SOS hierarchy developed by Lasserre \cite{lasserre}. Let $\textbf{y} = \{ y_\alpha \}_{\alpha \in \mathbb{N}^n}$ be an infinite real sequence and let $L_\textbf{y} : \mathbb{R}[\textbf{x}] \rightarrow \mathbb{R}$ be the Riesz linear functional defined by 
	\[
	f(\textbf{x}) = \sum_{\alpha \in \mathbb{N}^n} f_\alpha \textbf{x}^\alpha \mapsto L_\textbf{y}(f) = \sum_{\alpha \in \mathbb{N}^n} f_\alpha y_\alpha.
	\]
	Therefore, if $\textbf{y}$ is the moment sequence of a measure $\mu$ supported on a set $K$, i.e. 
	\[
	y_{\alpha} = \int_K \textbf{x}^\alpha \mathrm{d}\mu(\textbf{x})\, \text{ for } \alpha \in \mathbb{N}^n 
	\]
	then $L_{\textbf{y}}$ coincides with the integration operator on polynomials with respect to $\mu$, i.e. for a polynomial $f \in \mathbb{R}[\textbf{x}]$ we find
	\[
	L_{\textbf{y}}(f) = \sum_{\alpha \in \mathbb{N}^n} f_\alpha y_\alpha = \int_K f(\textbf{x})\mathrm{d}\mu(\textbf{x}).
	\]
	Given a finite sequence $\textbf{y} = \{y_\alpha\}_{\alpha \in \mathbb{N}^n_{2r}}$, we associate the so-called {\em truncated moment matrix} $M_r(\textbf{y})$ to $\textbf y$, defined as $(M_r(\textbf{y}))_{\alpha, \beta} = y_{\alpha+\beta}$ for $\alpha, \beta \in \mathbb{N}^n_r$. Such a matrix has dimensions $s(r) \times s(r)$, where
	
	\[
	s(r) = \binom{n+r}{r}.
	\]
	For $\textbf{y}$ as above, given a polynomial $g \in \mathbb{R}[\textbf{x}]$ of degree $d$, we define the \emph{localizing matrix} $M_r(g \star \textbf{y})$ associated to $\textbf{y}$ and $g$ as 
	\[
	(M_r(g \star \textbf{y}))_{\alpha, \beta} = \sum_{\gamma \in \mathbb{N}^n_d} g_\gamma y_{\alpha+\beta+\gamma}\,, \text{ for } \alpha, \beta \in \mathbb{N}^n_r.
	\]
	Let us now assume that $K$ is defined as a finite conjunction of polynomial inequalities, i.e., a {\em basic closed semialgebraic} set: 
	\begin{equation}\label{ksemialg}
	K = \{ \textbf{x} \in \mathbb{R}^n : h_i(\textbf{x})\ge 0 \text{ for } i \in [m_3]  \}\,.
	\end{equation}
	For later purpose, let us define 
	\[
	r_{\min} := \max_{i \in [m_1], j\in [m_2], k \in [m_3]}\{ \textrm{deg}(f_0), \textrm{deg}(f_i), \textrm{deg}(g_j),\textrm{deg}(h_k) \}
	\,. 
	\]
	For $r \in \mathbb{N}$, with $r \ge r_{\min}$, the level $r$ of Moment-SOS relaxation of \eqref{gengmp} is defined as  
	
	\begin{equation}\label{momsos}
	\begin{aligned}
	\inf_{\textbf{y} \in \mathbb{R}^{s(2r)+r_{\min}}} \; & L_\textbf{y}(f_0) \\
	\text{s.t.} \; & L_\textbf{y}(f_i) = a_i\,, \text{ for } i \in [m_1] \\
	& L_\textbf{y}(g_i) \le b_i\,, \text{ for } i \in [m_2] \\
	& M_r(\textbf{y}) \succeq 0 \\
	& M_r(h_i \star \textbf{y}) \succeq 0\,, \text{ for } i \in [m_3].	        
	\end{aligned}
	\end{equation}
	For each $r$ this is a semidefinite optimization problem (SDP) that can be solved up to arbitrary precision in polynomial time using for instance interior point methods. SDPs can be understood as a powerful generalization of linear programming problems (LPs), which are more common in practice. The difference is that the decision variables in SDPs are positive semidefinite matrices, whereas in LPs these are nonnegative scalar variables. For a comprehensive introduction in semidefinite optimization we refer to the paper by Boyd and Vandenberghe, see \cite{boyd}. The moment and localizing matrices depend linearly on $\textbf{y}$ and the cost is linear in $\textbf{y}$. The Moment-SOS hierarchy presented above was introduced by Lasserre, see \cite{lasserre}. For a survey on semidefinite programming relaxation of GMPs we refer the reader to \cite{deKlerk}. Also worth mentioning is the following sufficient condition for convergence of the Moment-SOS hierarchy to the optimal value of the corresponding GMP. For this we make the following assumption which is slightly stronger than compactness. 
	
	\begin{assumption}\label{archim}
		Let $h_j(\textbf{x})$ for $j \in [m_3]$ be the defining polynomials of $K$ in \eqref{ksemialg} and define $h_0(\textbf{x}) = 1$ for all $\textbf{x} \in \mathbb{R}^n$. There exist SOS polynomials $\sigma_j$ for $j = 0, 1, \dots, m_3$ such that $N - \|\textbf{x}\|^2 = \sum_{j = 0}^{m_3} \sigma_j(\textbf{x}) h_j(\textbf{x})$.
	\end{assumption}
	
	Assumption \ref{archim} is equivalent to the so-called \emph{Archimedian condition} and if it is satisfied, the Moment-SOS relaxation \eqref{momsos} converges to the optimal value of \eqref{gengmp} for $r \rightarrow \infty$, (cf. \cite[Theorem 4.1]{lasserre2}). Nie \cite{nie} proved that the optimal value is achieved for a finite $r$ for generic polynomial optimization problems. Note that if we know that $K$ is compact one can simply add the redundant constraint $N-\|\textbf{x}\|^2\ge 0$ to $K$ for $N$ such that $N \ge \|\textbf{x}\|^2$ for all $\textbf{x} \in K$ so that Assumption \ref{archim} is satisfied.

	\section{Bounds on options via the GMP formulation}\label{sec:outer}
	
	In this section we will cast the problem of computing bounds on the price of European call options as a particular instance of the GMP. The option will be dependent on $n$ assets $S_1, \dots, S_n$. We will denote the payoff function by $\varphi : \mathbb{R}^n_+ \rightarrow \mathbb{R}_+$, which may depend on the prices of the $n$ different assets. We assume the payoff is nonnegative, since we consider options, meaning there is no obligation of the owner to exercise it, in which case the payoff is zero. The range of possible prices for asset $S_i$ will be the nonnegative reals, i.e., $x_i \in \mathbb{R}_+$. Note that the payoff function is what defines the type of the option. As has been mentioned the no-arbitrage assumption is equivalent to the existence of a probability measure $\mu$ such that asset prices become martingales under $\mu$. This measure is referred to as the \emph{equivalent martingale measure} or the \emph{risk-neutral measure}. The price of the option is then given by the expectation of the payoff function with respect to this measure. Here and throughout this paper we assume for simplicity an interest rate of $0$.
	
	\subsection{Problem statement}
	
	For some finite index set $\mathcal{I}$ let information pairs $(f_i, q_i)$ for $i \in \mathcal{I}$ where $f_i : \mathbb{R}^n_+ \rightarrow \mathbb{R}$ and $q_i \in \mathbb{R}$, be given. These pairs might consist of payoff functions $f_i$ of options on the assets $S_1, \dots, S_n$ with the observable prices $q_i$ at which theses options are traded, or prices of derivatives on moments of underlying asset, such as mean, variance or correlation. In order to find bounds for the option at hand we will look for a probability measure that is consistent with this given information. In other words, the feasible set of measures $\mu$ will consist of measures such that
	
	\[
	\int_{\mathbb{R}^n_+}f_i(\textbf{x})\mathrm{d}\mu(\textbf{x})\leqq q_i\,, \text{ for all } i \in  \mathcal{I},
	\] 
	where $"\leqq"$ means either $"\le"$ or $"="$. We will also assume that the $d$-th order moments of the corresponding distributions are finite for some $d \in \mathbb{N}$. To fix ideas we will consider the following problem adapted from \cite{bertsimas}. Given $n$ assets $S_1, \dots, S_n$ whose prices are given by $x_1, \dots, x_n$, we want to find a lower bound on a European call option whose payoff may depend on the assets $S_i$ for $i \in [n]$. The available information  is the following: we have $N_i \in \mathbb{N}$ prices $q_{i,j}, j \in [N_i]$ of options on asset $S_i$ with strikes $k_{i,j}$ for $j \in [N_i]$.	Additionally, we have some moment information in the following form 
	\[
	\int_{\mathbb{R}^n_+} f_\ell(\textbf{x})\mathrm{d}\mu(\textbf{x}) = p_\ell\,,
	\]
	where  $ f_\ell \in \mathbb{R}[\textbf{x}]$ and $p_\ell \in \mathbb{R}$. For example, if $\gamma_i$ is the observed mean of asset $i$ and the observed covariance of assets $i$ and $j$ is $\sigma_{i,j}$, one can add the constraint
	\[
	\int_{\mathbb{R}^n_+} (x_i-\gamma_i)(x_j-\gamma_j)\mathrm{d}\mu(\textbf{x}) = \sigma_{i,j}\,.
	\]
	
	Further, we assume the $d$-th order moments under a risk-neutral pricing measure are finite, where 
	\[ 
	d =\max_{i \in [n], j\in [N_i], \ell \in [m]}\{ \textrm{deg}(\varphi), \textrm{deg}(f_{i,j}), \textrm{deg}(f_\ell) \} +1.
	\]
	What we mean by this is that 
	\[
	\int_{\mathbb{R}^n_+} \|\textbf{x}\|_2^d \mathrm{d}\mu(\textbf{x}) \le M
	\]
	for some $M \in \mathbb{R}_+$, where $\|\textbf{x}\|_2 = \sqrt{x_1^2 + \dots + x_n^2}$ is the standard Euclidean $\ell_2$-norm. A risk-neutral pricing measure is a measure such that the asset prices are equal to the expectation under this measure discounted by the risk-free interest rate. For convenience, we assume that $d$ is even, otherwise we set $d \leftarrow d+1$. This way we make sure that we are dealing with a GMP with (piecewise) polynomial data. 
	The optimal value of the optimization problem below will serve as bound for the given option that is consistent with the available information.
	
	\begin{equation} \label{prob1}
	\begin{aligned}
	\sup_{\mu \in \mathcal{M}(\mathbb{R}^n_+)_+} / \inf_{\mu \in \mathcal{M}(\mathbb{R}^n_+)_+} \; & \int_{\mathbb{R}^n_+} \varphi(\textbf{x}) \mathrm{d}\mu(\textbf{x}) \\
	\text{ s.t. } &\int_{\mathbb{R}^n_+} \max(0, x_i-k_{i,j}) \mathrm{d}\mu(\textbf{x}) = q_{i,j}\,, \text{ for } i \in [n], j\in [N_i]  \\
	& \int_{\mathbb{R}^n_+}f_\ell (\textbf{x}) \mathrm{d}\mu(\textbf{x}) = p_\ell\,,  \text{ for } \ell \in [m] 
	\\
	&\int_{\mathbb{R}^n_+} \mathrm{d}\mu(\textbf{x}) =1 
	\\
	&\int_{\mathbb{R}^n_+} \|\textbf{x}\|_2^d \mathrm{d}\mu(\textbf{x}) \le M\,.
	\end{aligned}
	\end{equation}
	To obtain upper bounds we maximize and for lower bounds we minimize. In a nutshell, one is looking for the probability distribution of the asset price, that is consistent with the known information and minimizes (respectively maximizes) the objective. 
	
	\subsection{Existence of an optimal solution}

	Now we prove that the infimum/supremum in~(\ref{prob1}) is attained. In order to do so, we will use the Prokhorov theorem \cite{prokhorov} asserting a weak sequential compactness of a family of tight measures.
	\begin{definition}[Tightness]
		A sequence of measures $(\mu_k)_{k=1}^\infty$ defined on $\mathbb{R}^n$ is called tight if for every $\epsilon > 0$ there exists a compact set $K$ such that $\mu_k(K^c) < \epsilon$ for all $k \in \mathbb{N}$.
	\end{definition}
	
	\begin{theorem}[Prokhorov]\label{Thmprokhorov}
		Let $(\mu_k)_{k=1}^\infty$ be a tight sequence of Borel probability measures on $\mathbb{R}^n$. Then there exists a Borel probability measure $\mu$ and a subsequence $(\mu_{k_i})_{i=1}^\infty$ converging weakly to $\mu$, i.e.,
		\begin{equation}\label{weakConv}
		\lim_{i\to\infty}\int g \,\mathrm{d}\mu_{k_i} = \int g\,\mathrm{d}\mu
		\end{equation}
		for all bounded continuous functions $g$ on $\mathbb{R}^n$.
	\end{theorem}

	\begin{lemma}\label{infattained}
		If Problem~(\ref{prob1}) is feasible, then its supremum/infimum  is attained.
	\end{lemma}
	\begin{proof}{}
		We begin by observing that if (\ref{prob1}) is feasible, then the infimum in (\ref{prob1}) is finite since the objective function is nonnegative. 
		Also, the supremum is finite because of the last constraint $\int_{\mathbb{R}^n_+} \| \textbf{x}\|_2^d \mathrm{d}\mu(\textbf{x}) \le M$. Denote $f_{i,j}:= \max(0,x_i-k_{i,j})$ and let $(\mu_k)_{k=1}^\infty$ be an  optimizing sequence for (\ref{prob1}). Denote by $\phi_k$ the measures defined by
		\[
		\mathrm{d}\phi_k = (1+\|\textbf{x}\|_2^{d-1})\mathrm{d}\mu_k\,. 
		\]
		Moving on, we show that the sequence $(\phi_k)_{k=1}^\infty$ is tight. Let $\epsilon >0$ be given and let $K$ be the closed ball of radius $a$. Then we have
		\begin{equation*}
		\begin{aligned}
		\phi_k(K^c) &= \int_{\mathbb{R}^n} \mathbbm{1}_{\{ \|\textbf{x} \|_2 \ge a \}}(1+\|\textbf{x}\|_2^{d-1})\mathrm{d}\mu_k \\ 
		&\le \frac{1}{a} \int_{\mathbb{R}^n}\|\textbf{x}\|_2  (1+\|\textbf{x}\|_2^{d-1})\mathrm{d}\mu_k \le \frac{M^{1/d} + M}{a}\,,
		\end{aligned}    
		\end{equation*}
		where we used Jensen's inequality \cite{Jensen1906} in the last step. By picking $a$ sufficiently large, we make $\phi_k(K^c) < \epsilon$, hence establishing tightness. By Theorem \ref{Thmprokhorov}, there exists a weakly convergent subsequence (that we do not relabel) that converges weakly to a measure $\phi$. We set
		\[
		\mathrm{d}\mu := \frac{\mathrm{d}\phi }{1+\|\textbf{x}\|_2^{d-1}}
		\]
		to be the candidate optimizer for~(\ref{prob1}). We first show that the equality constraints for (\ref{prob1}) are satisfied by $\mu$. We have
		\begin{align*}
		q_{i,j} & = \lim_{k\to\infty}\int f_{i,j}\,\mathrm{d}\mu_k = \lim_{k\to\infty} \int \frac{f_{i,j}}{1+\|\textbf{x}\|_2^{d-1}}\,\mathrm{d}\phi_k \\
		&= \int \frac{f_{i,j}}{1+\|\textbf{x}\|_2^{d-1}}\,\mathrm{d} \phi = \int f_{i,j}\,\mathrm{d}\mu\,,
		\end{align*}
		where in the third equality we used the fact that the function $\frac{f_{i,j}}{1+\|\textbf{x}\|_2^{d-1}}$ is continuous and bounded. The same argument applies to the objective function and the constraint $\int \mathrm{d}\mu = 1$, as well as for the functions $f_\ell, \ell \in [m]$. 
		Finally, we establish that $\int \|\textbf{x}\|_2^{d}\,\mathrm{d}\mu < M$. We define $f_n(\textbf{x}) := \min(\|\textbf{x}\|_2^d,n)$. Then we have
		\begin{align*}
		\int \|\textbf{x}\|_2^d\,\mathrm{d}\mu &\overset{(i)}{=} \lim_{n\to\infty} \int  f_n\,\mathrm{d}\mu = \lim_{n\to\infty} \int  \frac{f_n}{1 + \|\textbf{x}\|_2^{d-1}}\,\mathrm{d}\phi \\  &\overset{(ii)}{=}  \lim_{n\to\infty} \lim_{k\to\infty} \int  \frac{f_n}{1 + \|\textbf{x}\|_2^{d-1}} \,\mathrm{d}\phi_k  = \lim_{n\to\infty} \lim_{k\to\infty} \int  f_n \,\mathrm{d}\mu_k \\ &\overset{(iii)}{\le} \lim_{n\to\infty} \lim_{k\to\infty} \int  \|\textbf{x}\|_2^d \,\mathrm{d}\mu_k \le M\,,
		\end{align*}
		where we used the monotone convergence theorem in $(i)$, the weak convergence of $\phi_k$ to $\phi$ in $(ii)$ and the fact that $f_n \le \|\textbf{x}\|_2^d$ in $(iii)$. $\hfill \square$
	\end{proof}
	
	Combining this result with the Richter theorem (see \cite[Satz 4]{richter} for an original reference or \cite[Theorem 19]{diDio2018} for a modern statement and historical remarks), we get the following immediate corollary.
	\begin{corollary}\label{cor1}
		If Problem~(\ref{prob1}) is feasible, then the optimal value of~(\ref{prob1}) is attained by an atomic measure with finitely many atoms (at most $n\sum_{i=1}^n N_i+m+3$).
	\end{corollary}
	
	We finish this section by showing that finite $d$-th order moments are necessary for the existence of an optimal solution. 
	
	\begin{proposition}
		The last constraint in \eqref{prob1} cannot be omitted in Lemma \ref{infattained}.
	\end{proposition}
	
	\begin{proof}{}
		Consider the following problem
		
		\begin{equation}\label{exampleGMP3}
		\begin{aligned}
		p^\ast  = \inf & \int_{0}^\infty \max(0, x-k_1)\mathrm{d}\mu  \\
		\text{s.t.} & \int_{0}^\infty \max(0,x-k_2)\mathrm{d}\mu = a \\
		&\int_{0}^\infty \mathrm{d}\mu = 1 \,, \\
		\end{aligned}
		\end{equation}
		where we assume $k_1 < k_2$ and $a \neq 0$. Note that this implies that for the optimal value we have $p^\ast \ge a$. We will show that there exists no measure for which the optimal value is attained. The following is a minimizing sequence for \eqref{exampleGMP3}
		\[ \mu_n = \left(1-\frac{1}{n}\right)\delta_{k_1}+\frac{1}{n}\delta_{k_2+na}\,. \]
		For every $n \in \mathbb{N}$ we see that $\mu_n$ is a probability measure  as it is a convex combination of atomic measures and 
		
		\[ \int_0^\infty \max(0,x-k_2)\mathrm{d}\mu_n = \frac{1}{n}\left( k_2+na - k_2 \right) = a\,. \]
		So the sequence is indeed feasible. For the objective value we get 
		
		\[ \int_0^\infty \max(0,x-k_1)\mathrm{d}\mu_n = \left(1-\frac{1}{n}\right)(k_1-k_1)+\frac{1}{n}\left( k_2+na -k_1 \right) = a + \frac{1}{n}(k_2-k_1)\,. \]
		So we have that $\mu_n$ is a minimizing sequence as it is feasible and converges to $a \le p^\ast$. The limit $\lim_{n \rightarrow \infty} \mu_n = \delta_{k_1}$, however, is not feasible. We now show that there exists no probability measure $\mu \in \mathcal{M}(\mathbb{R}_+)_+$ that is optimal for \eqref{exampleGMP3}. For this we assume that $\mu$ is an optimizer of \eqref{exampleGMP3}. Then we have 
		
		\[ 
		\int_{\mathbb{R}_+} \max(0,x-k_1)\mathrm{d}\mu(x) = a = \int_{\mathbb{R}_+} \max(0,x-k_2)\mathrm{d}\mu(x) \,.
		\]
		Thus, 
		
		\begin{equation*}
		\begin{aligned}
		0 &=	\int_{\mathbb{R}_+} \max(0,x-k_1)\mathrm{d}\mu(x) - \int_{\mathbb{R}_+} \max(0,x-k_2)\mathrm{d}\mu(x) \\
		& = \int_{k_1}^{k_2}\underbrace{(x-k_1)}_{\ge 0} \mathrm{d}\mu(x) + \int_{k_2}^{\infty} \underbrace{(k_2-k_1)}_{>0}\mathrm{d}\mu(x)\,.
		\end{aligned}
		\end{equation*}
		The latter integral must be zero which implies that $\textrm{supp}(\mu) \cap [k_2,\infty) = \emptyset$. But if that is the case we have 
		\[ 
		a = \int_{\mathbb{R}_+} \max(0,x-k_2)\mathrm{d}\mu(x) = \int_{k_2}^{\infty}(x-k_2)\mathrm{d}\mu(x) = 0.
		\]
		Therefore, $\mu$ cannot be feasible. $\hfill \square$
	\end{proof}
	
	This example therefore illustrates that the support of a minimizing sequence can tend to infinity.

	\section{Bounding the support}\label{sec:boundsupport}
	
	By Corollary~\ref{cor1} the optimal solution to \eqref{prob1} is a measure with finitely many atoms. This section is devoted to the question whether it is possible to bound the support of the optimal solution in terms of the problem data of \eqref{prob1}. If this were possible, i.e., if we knew the optimal solution is attained in a box $[0,B]^n$ for some $B \in \mathbb{R}_+$, we could consider a compact version of \eqref{prob1}, where $\mathbb{R}^n_+$ is replaced by $[0,B]^n$. 
	This has the advantage that we know that the Moment-SOS hierarchy converges if the underlying sets are compact (recall Assumption~\ref{archim} in connection with \cite[Theorem 4.10]{lasserre2}).
	It may be possible to derive such a bound by analyzing the core variety associated to the moment functional arising from any optimal solution of (\ref{prob1}). Specifically, in view of Theorem~2.10 of~\cite{blekherman} one should bound a $ B \in \mathbb{R}_+$ such that the core variety associated to the set $[0,B]^n$ and the optimal moment functional of (\ref{prob1}) is nonempty. We carry out the core variety procedure for an artificial example to show how it works and demonstrate that in special cases it is possible to derive a bound on the support in this way. 
	
	\subsection{Approach 1: Core variety}
	
	We follow the notation of Section 1.1. of \cite{blekherman}. 
	Consider the following problem 
	\begin{equation}
	\begin{aligned}
	p^\ast  = \inf & \int_{0}^\infty \max(0, x-k)\mathrm{d}\mu \\
	\text{s.t.} &\int_{0}^\infty \mathrm{d}\mu = 1 \\
	&\int_{0}^\infty x^2\mathrm{d}\mu \le M \,, \\
	\end{aligned}
	\end{equation}
	for some $M >0$. Let $a$ be the objective value of this problem for some feasible measure $\mu^\ast$ and let $m$ be such that we have $\int_{0}^\infty x^2\mathrm{d}\mu^\ast = m \le M$. Define $f = \max(0,x-k)$.
	Let $S$ be the interval $[0,B]$ with $B > 0$ to be determined and let
	\[
	V := \mathrm{span}\{1,f,x^2\}\,.
	\]
	Note that the constant function must be included since we are looking for representing probability measures. Define the linear functional $L: V\to \mathbb{R}$ by
	\[
	L(c_1\cdot 1 + c_2f + c_3x^2) = \int (c_1+c_2f + c_3 x^2) \,\mathrm{d}\mu = c_1 + c_2a + c_3m \,,	\]
	for $(c_1,c_2,c_3) \in \mathbb{R}^3$ and define
	\[
	S_0 := S\,,
	\]
	which is the initial step of the core-variety iterative computation procedure. The next  step is given by
	setting 
	\[
	S_1 = \mathcal{Z}(g \in V \mid L(g) = 0, g \ge 0\; \mathrm{on}\;S_0  )\,,
	\]
	where $\mathcal{Z} (P)$ denotes the set of all common zeros of the functions contained in $P$. The subsequent steps of the core variety computation procedure are given by induction and the core variety itself is the terminal step of this procedure (which is provably finite). Here we prove that $S_1 = S_0 = [0,B]$ whenever $B > k$ if $a = 0$ and $B > \frac{M+\sqrt{M(M-4ak)}}{2a}$ if $a >0$. In order to do so, let $g = c_0 + c_1f +c_2x^2\in V$ be given. The requirement of $L(g) = 0$ means that
	\[
	c_0 + c_1 a +c_2m = 0
	\]
	and hence $c_0 = -c_1 a-c_2m$. Therefore
	\[
	g = c_1(f - a)+c_2(x^2-m) = c_1(\max(0,x-k) - a)+c_2(x^2-m)\,.
	\]
	We need to understand when $g(x) \ge 0$ for all $x \in [0,B]$, i.e., what restriction do we have on the $c_i$'s for $i = 1,2$. 
	
	We distinct two main cases, $a=0$ and $a >0$ and then consider subcases to solve the problem. The aim is to determine a $B \in \mathbb{R}_+$ such that all $g \in V$ satisfying $L(g) = 0$ and $g \ge 0$ on $[0,B]$ are identically zero on $S_0$. This is the case if $c_0 = c_1 = c_2 = 0$ and the core variety procedure terminates.

	\textbf{Case 1:} $a=0$. First, note that if $a = 0$ we have $\sqrt{m} \le k$. The reason is that 
	\[ 
	0 = \int_{0}^\infty \max(0,x-k)\mathrm{}d\mu = \int_{k}^{\infty} (x-k) \,\mathrm{d}\mu
	\]
	and so $\mathrm{supp}(\mu) \cap (k, \infty) = \emptyset$, which implies 
	\[
	m = \int_{0}^{\infty}x^2 \mathrm{d}\mu = \int_{0}^{k}x^2 \mathrm{d}\mu \le k^2 \int_{0}^{k} \mathrm{d}\mu = k^2
	\]
	and so $\sqrt{m} \le k$.
	
	\textbf{Case 1.1. $c_2 >0$.}
	Since $f(0) = 0$, we have $g(0) = -(c_2m)$ and hence $c_2$ cannot be positive.

	\textbf{Case 1.2. $c_1 < 0, c_2 < 0$.}
	Now, for $x > \max(\sqrt{m},k) = k$, both $(x^2-m)$ and $(\max(0,x-k))$ are strictly positive and so $c_1, c_2$ cannot be strictly negative at the same time, since then $g(B)<0$ for $B>k$. 
	
	\textbf{Case 1.3. $c_1>0, c_2<0$.}
	If $\sqrt{m} =k$ set $x>k$.
	If $\sqrt{m} < k+a$ we can simply set $x = \sqrt{m}$ and see that
	\[
	0 \le g(x) = c_1(x-k)+c_2(x^2-k^2) = (x-k)(c_1 + c_2(x+k))\,,
	\]
	which becomes negative if $x \ge -\frac{c_1}{c_2}-k$. Assume $0 \le g(k+\varepsilon)$ for some $\varepsilon > 0$. Then, 
	\[
	0 \le \varepsilon c_1 + c_2(2\varepsilon k + \varepsilon^2) 
	\]
	and so 
	\[
	- \frac{c_1}{c_2} \le 2k+\varepsilon\,,
	\]
	from which follows that the choice $c_1 >0, c_2 <0$ leads to $g(x)<0$ if $x > k+\varepsilon$ for any $\varepsilon > 0$. If instead $\sqrt{m} < k$ we can simply set $x=k$ to find 
	\[
	0 \le g(k) = c_2(k^2-m) < 0\,. 
	\]
	
	For all cases above we found that for $B >k$, the only function $g \in V$ that satisfies $g(x)\ge0$ on $S_0$ is identically $0$. \par
	\textbf{Case 2: $a>0$.}
	In this case we must again check all possibilities for $c_1,c_2$.

	\textbf{Case 2.1: $c_1>0,c_2>0$.}
	For $x = 0$ we see
	\[
	0 \le g(0) = c_1(-a)+c_2(-m) < 0\,.
	\]
	\textbf{Case 2.2: $c_1<0,c_2<0$.}
	For $x \ge \max(\sqrt{m}, k+a)$ we find
	\[
	0 \le g(x) = c_1(x-k-a)+c_2(x^2-m) < 0\,.
	\]
	\textbf{Case 2.3: $c_1>0,c_2<0$.} If $\sqrt{m} < k+a$ we find	
	\[
	g(\sqrt{m}) = c_1 \underbrace{(\max(0, \sqrt{m}-k)-a)}_{<0}+c_2((\sqrt{m})^2-m) <0\,.
	\]
	If $\sqrt{m} \ge k+a$, then note that $0 \le g(0) = c_1 (-a)+c_2(-m)$ from which follows $c_1 \le -c_2 \frac{m}{a}$. Then for $x \ge \sqrt{m}$ we find 
	\[ 
	0 \le c_1(x-k-a)+c_2(x^2-m) \le c_2\left(x^2-\frac{m}{a}x+\frac{mk}{a}\right)\,. 
	\]
	The content of the brackets is positive for 
	\[
	x > \frac{m+\sqrt{m(m-4ak)}}{2a}\,,
	\]
	and note the term under the square root is positive because $m \ge k^2+2ak+a^2$.
	
	\textbf{Case 2.4. $c_1<0,c_2>0$.}
	Let $\sqrt{m} > k+a$. Setting $x = k+a$ we find $g(k+a) = c_2((k+a)^2-m)<0$. Thus, consider the case where $\sqrt{m} \le k+a$. 
	Note that we can deduce $\frac{c_1}{c_2} \le - \frac{m}{a}$ from $g(0)\ge0$. Let $x > k$. We want to check for what $x$ we have
	$c_1(x-k-a)+c_2(x^2-m)<0$. This is the case if 
	\[
	\frac{c_1}{c_2} \le -\frac{m}{a} < -\frac{x^2-m}{x-k-a}\,.
	\]
	We are looking for the smallest root $x^\ast$ of $x^2-\frac{m}{a}x+\frac{mk}{a}$ such that $x^\ast>k$. The roots are given by 
	\[
	x_{1,2} = \frac{m\pm \sqrt{m(m-4ak)}}{2a}\,.
	\] 
	We show
	\[
	x_1 = \frac{m - \sqrt{m(m-4ak)}}{2a} > k\,.
	\]
	For this note that 
	\[
	\frac{m - \sqrt{m(m-4ak)}}{2a} > k \quad \Leftrightarrow  \quad m - 2ak > \sqrt{m^2-4amk}\,.
	\]
	Squaring both sides and cleaning up we see this is true.
	
	In conclusion, if $B > \frac{M+\sqrt{M(M-4ak)}}{2a}$ the core variety procedure terminates after the first step and the support of the corresponding measure lies in $[0,B]$.

	\subsection{Approach 2: Atomic representation}
	
	We now present a different approach to the problem of bounding the support to verify the bound we obtained before. Consider the univariate, i.e. $n = 1$ case for problem \eqref{prob1} and assume it is feasible and $m = 0$. Also let the strike prices be ordered, i.e. $k_1 \le k_2 \le \dots \le k_{N_1}$. Therefore we have $N_1$ equality constraints, each corresponding to the observable price of a vanilla option on the considered asset. We further assume the payoff function to be the payoff of a European call option, i.e. $\varphi(x) = \max(0, x-k)$ for some $k \in \mathbb{R}_+$. By Corollary \eqref{cor1} there exists an atomic solution of the form $\sum_{j = 1}^{m}\alpha_j \delta_{x_j}$. Let $0 \le x_1 \le x_2 \le \dots \le x_m$. The following lemma shows that we may assume w.l.o.g.~that $x_{m-1} \le k_{N_1} \le x_{m}$.
	
	\begin{lemma}\label{location}
		Consider \eqref{prob1} for $n = 1$ and $m = 0$. If there exists an optimal solution, then there exists one such that exactly one atom $x^{(i)}$ lies in $(k_{N_1}, \infty)$. Moreover, there exists a solution such that in each of the intervals 
		\[ 
		[0,k_1], [k_1,k_2], \dots, [k_{N_1-1},k_{N_1}], [k_{N_1}, \infty)
		\]
		there is at most one atom.
	\end{lemma}
	
	\begin{proof}{}
		For the first claim asserting that there exists an optimal measure $\mu^\ast$ such that exactly one atom lies in $(k_{N_1}, \infty)$, let us assume that all atoms lie in $[0,k_{N_1}]$. Then 
		
		\[ 
		a_{N_1} = \int_{\mathbb{R}_+} \max(0,x-k_{N_1})\mathrm{d} \mu^\ast = \int_{0}^{k_{N_1}} \max(0,x-k_{N_1})\mathrm{d} \mu^\ast =  0\,, 
		\]
		which is a contradiction. Thus, at least one atom lies in $(k_{N_1}, \infty)$. Suppose two atoms lie in $(k_{N_1}, \infty)$ and let the associated weighted Dirac measures be $\alpha \delta_{x_1}$ and $\beta \delta_{x_2}$ with $\alpha, \beta > 0$ and $x_1, x_2 > k_{N_1}$. Now, since $x_1,x_2 > k_{N_1}$ these two Dirac measures influence every constraint of \eqref{prob1} as well as the objective because all input functions are strictly positive at $x_1, x_2$. Their influence is exactly 
		\[ 
		\alpha (x_1-k_i)+\beta(x_2-k_i) = (\alpha+\beta)\left( \frac{\alpha}{\alpha + \beta}x_1+\frac{\beta}{\alpha+\beta}x_2-k_i\right)\,. 
		\]
		It follows that these two Dirac measure can be combined to a single one with weight $\omega = \alpha + \beta > 0$ and support $x =  \frac{\alpha}{\alpha + \beta}x_1+\frac{\beta}{\alpha+\beta}x_2> k_{N_1}$ without changing the influence on the data. Also, because $\|\cdot\|^d$ is convex, the inequality constraint is also satisfied. By similar reasoning one can prove the second claim of the lemma. $\hfill \square$
	\end{proof}

	We deduce that therefore $\alpha_m (x_m-k_{N_1}) = a_{N_1}$. We also know $\alpha_m x_m^2 \le M \Leftrightarrow \alpha_m \le M/x_m^2$, from which follows that $a_{N_1} \le M/x_m^2(x_m-k_{N_1})$. Hence, 
	\[
	x_m \le \frac{M+\sqrt{M(M-4a_{N_1}k_{N_1})}}{2a_{N_1}} =: B\,.
	\]
	In the univariate case the support of an optimal solution lies in $[0,B]$. 
	
	\begin{remark}
		We would like to remark that the model we consider is still NP-hard in general. 
		There is no shift of complexity to the task of finding $B$. In comparison, in \cite{bertsimas} the authors propose to relax the set of feasible measures from $\mathcal{M}(\mathbb{R}^n_+)$ to $\mathcal{M}(\mathbb{R}^n)$, i.e., the
		set of Martingale measures over $\mathbb{R}^n_+$ vs. $\mathbb{R}^n$. They henceforth 
		identify cases in which the relaxation may be solved in polynomial time, and they state that the 
		result is not necessarily an optimal bound to the original problem. Our approach on the other hand is proven to 
		converge to the optimal bound. To actually compare the two approaches we first point out that Bertsimas and 
		Popescu do not introduce a hierarchy, but a single relaxation of the problem. The relaxation they propose could 
		be tackled using the standard Moment-SOS hierarchy. However, since they relax the local non-negativity to global non-negativity, the optimal solution to any level of the Moment-SOS hierarchy 
		applied to their relaxation will always be contained in the set of feasible solutions belonging to our approach. 
		Hence, our bound will always lie at least as close to the optimal bound as theirs. 
	\end{remark}

	\section{Examples for outer range}\label{sec:examplesOuter}
	
	We will now present some examples of numerical computations of bounds on option prices in the framework specified in the previous sections. The Moment-SOS hierarchy provides a lower bound to the minimization problem and an upper bound to the maximization problem, which is why we call these outer bounds. The implementation was coded in the Julia programming language and we used the MOSEK solver \cite{mosek} version 9.1.9. The code is available online\footnote{\url{ https://github.com/FelixKirschner/boundingOptionPricesCode}} and relies partly on the Julia package MomentOpt.jl \cite{momOpt}.

	\subsection{Univariate case}
	Let us describe our implementation strategy for the univariate case. Assume we want to find bounds on the price of an option with strike $k$ given strikes and prices of other options on the same asset, i.e., the following problem:
	
	\begin{equation} \label{numexp}
	\begin{aligned}
	\sup_{\mu \in \mathcal{M}(\mathbb{R}_+)_+} / \inf_{\mu \in \mathcal{M}(\mathbb{R}_+)_+} \; & \int_{\mathbb{R}_+} \max(0, x-k) \mathrm{d}\mu(x) \\
	\text{ s.t. } &\int_{\mathbb{R}_+} \max(0, x-k_{i}) \mathrm{d}\mu(x) = a_i\,, \text{ for } i \in [n]   \\
	& \int_{\mathbb{R}_+} \mathrm{d}\mu(x) = 1  \\
	& \int_{\mathbb{R}_+} x^2 \mathrm{d}\mu(x) \le M\,.
	\end{aligned}
	\end{equation}
	Since we know from Lemma~\ref{infattained} that feasibility implies the existence of an optimal solution, we will assume the optimal solution will be attained in a box $[0,B]$ for some $B \in \mathbb{R}$. A suitable $B$ can be obtained via the procedure described in section~\ref{sec:boundsupport}.\\
	To circumvent the problem of dealing with piecewise affine functions we split the interval $[0,B]$ into subintervals and define measures supported on each of the subintervals. For this let $m$ be the index such that $k_m < k < k_{m+1}$. We define intervals $[0,k_1], [k_i,k_{i+1}]$ for $i = 1, \dots, m-1$, as well as $[k_m,k], [k,k_{m+1}]$ and $[k_j,k_{j+1}]$ for $j=m+1,n-1$ and finally $[k_n,B]$. The situation is visualized in Figure~\ref{fig:univariateSlices}
	\begin{figure}
		\centering
		\begin{tikzpicture}[thick,scale=0.9, every node/.style={scale=0.8}]\label{fig:univariateSlices}
		\draw (1,1) -- (4.5,1); 
		\draw (5.5,1) -- (11.5,1); 
		
		\draw[->] (12.5,1) -- (15,1); 
		\draw (1,0.8) -- (1,1.2);
		
		\node at (5,1) {$\dots$};
		\node at (12,1) {$\dots$};
		
		\draw[dashed] (1,1.2) -- (1,2.5);
		
		\node at (1,0.5) {$0$};
		
		\node at (1.55, 2) {$[0,k_1]$}; 
		
		\draw (2.15,0.8) -- (2.15,1.2);
		\node at (2.15,0.5) {$k_1$};
		
		\draw[dashed] (2.15,1.2) -- (2.15,2.5);
		
		\node at (2.75, 2) {$[k_1,k_2]$};
		
		\draw (3.35,0.8) -- (3.35,1.2);
		\node at (3.35,0.5) {$k_2$};
		
		 \draw[dashed] (3.35,1.2) -- (3.35,2.5);
		
		\node at (5, 2) {$[k_2,k_3] \dots [k_{m-1},k_m]$};
		
		\draw (6.75,0.8) -- (6.75,1.2);
		\node at (6.75,0.5) {$k_m$};
		
		\draw[dashed] (6.75,1.2) -- (6.75,2.5);
		
		\node at (7.5, 2) {$[k_m,k]$};
		
		\draw (8.25,0.8) -- (8.25,1.2);
		\node at (8.25,0.5) {$k$};
		
		\draw[dashed] (8.25,1.2) -- (8.25,2.5);
		
		\node at (9.175, 2) {$[k,k_{m+1}]$};
		
		\draw (10,0.8) -- (10,1.2);
		\node at (10,0.5) {$k_{m+1}$};
		
		\draw[dashed] (10,1.2) -- (10,2.5);
		
		\node at (12,2) {$[k_{m+1},k_{m+2}] \dots [k_n,B]$};
		
		\draw (14,0.8) -- (14,1.2);
		\node at (14,0.5) {$B$};
		
		\draw[dashed] (14,1.2) -- (14,2.5);
		\end{tikzpicture}
		\caption{Visualization of the segmentation of the interval $[0,B]$}
	\end{figure}
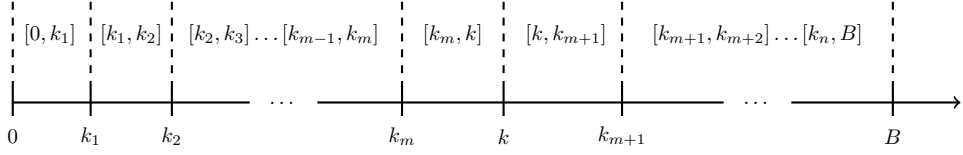

	Let $S$ be the collection of these subsets. Elements in $S$ are pairwise disjoint and the union of all sets in $S$ is $[0,B]$. The collection $S$ contains $n+2$ intervals and to each one is assigned a measure $\mu_i$ for $i = 1, \dots, n+2$. This way we can formulate a problem equivalent to \eqref{numexp}.
	
	\begin{equation} \label{numexp2}
	\begin{aligned}
	\sup / \inf \; & \sum_{i = m+2}^{n+2} \int (x-k) \mathrm{d}\mu_i(x) \\
	\text{ s.t. } &\sum_{i=j+2}^{n+2} \int (x-k_{j}) \mathrm{d}\mu_i(x) = a_j\,, \text{ for } j = m+1, \dots, n  \\
	&\sum_{i=j+1}^{n+2} \int (x-k_{j}) \mathrm{d}\mu_i(x) = a_j\,,\text{ for } j \in [m] \\
	& \sum_{i=1}^{n+2} \int x^2\mathrm{d}\mu_i(x) \le M  \\
	& \sum_{i=1}^{n+2} \int\mathrm{d}\mu_i(x) = 1 \,. \\
	\end{aligned}
	\end{equation}
	
	Let $s_i = [s_{i_1},s_{i_2}]$ for $i = 1, \dots, n+2$ be the elements of $S$. Introduce a linear operator $L_i^r : \mathbb{R}[x]_{2r} \rightarrow \mathbb{R}$ for every $\mu_i$. The level $r$ relaxation is then given by 
	\begin{equation} \label{relax1}
	\begin{aligned}
	\sup / \inf \; & \sum_{i = m+2}^{n+2} L_i^r(x-k) \\
	\text{ s.t. } &\sum_{i=j+2}^{n+2} L_i^r( x-k_{j} ) = a_j\,, \text{ for } j = m+1, \dots, n   \\
	&\sum_{i=j+1}^{n+2} L_i^r(x-k_{j}) = a_j\,, \text{ for } j\in[m]   \\
	& \sum_{i=1}^{n+2} L_i^r(x^2) \le M  \\
	& \sum_{i=1}^{n+2} L_i^r(1) = 1  \\
	& L_i^r([x]_r[x]_r^T) \in \mathcal{DNN} \,, \text{ for } i \in [n+2]  \\
	& L_i^r((s_{i_2}-x)(x-s_{i_1})[x]_{r-1}[x]_{r-1}^T) \in \mathcal{DNN} \,, \text{ for } i \in [n+2]\,,
	\end{aligned}
	\end{equation}
	where $\mathcal{DNN}$ is the doubly nonnegative cone, i.e.,  $\mathbb{S}^n_+ \cap \mathbb{R}_+^{n \times n}$ and the operator $L_i^r$ is applied entry-wise to the matrices $[x]_r[x]_r^T$. The decision variables here are the linear operators $L_i^r$. By introducing a variable $y_\alpha^{(i)} = L_i^{r}(\textbf{x}^\alpha)$, for $i \in [n+2], \alpha \in \mathbb{N}^n_r$ problem~\eqref{relax1} becomes a regular semidefinite program. For the actual calculation it is expedient to normalize everything, divide the given data by $B$. Consider problem~\eqref{numexp} with the data displayed in Table~\ref{tab:Table1} and with $k = 105$ and with $M = 200\,000$. Using the relaxation given in \eqref{relax1} we can approximate the optimal solution and we find the first level is tight, meaning we obtained the optimal bounds proposed by Bertsimas and Popescu in \cite{bertsimas}. For the considered case we get a lower bound of $3.875$ and an upper bound of $5.125$ and the computation took $0.01$ seconds.
	
	The domain in this problem is partitioned into $7$ parts. For each part we define a measure for each of which we introduce moment variables $y_{\alpha}^{(i)}$ for $\alpha \in \mathbb{N}^n_{2r+d_{\max}}$, where 
	\[
		d_{\max} = \max_{i \in [n], j\in [N_i], \ell \in [m]}\{ \textrm{deg}(\varphi), \textrm{deg}(f_{i,j}), \textrm{deg}(f_\ell) \}.
	\]
	Thus, for this particular problem, we introduced $7 \times 5 = 35$ variables. The number of involved matrices was $7 \times 2 = 14$, each of size $2 \times 2$. In total we had $14$ linear matrix inequality (LMI) constraints, $6$ equality constraints as well as $1+35 = 36$ inequality constraints, one to ensure finite $d$-th order moments and one for each variable to ensure $y_{\alpha}^{(i)}\ge 0$.
	\begin{table}
		\centering
		\begin{tabular}{|c|c|c|c|c|c|}
			\hline
			$i$ & $1$ & $2$ & $3$ & $4$ & $5$ \\
			\hline
			$k_i$ & $95$ & $100$ & $110$ & $115$ & $120$ \\
			\hline
			$a_i$ & $12.875$ & $8.375$ & $1.875$ & $0.625$ & $0.25$ \\
			\hline
		\end{tabular}
		\caption{Prices of European call options on the Microsoft stock from July '98 with strikes $k_i$}
		\label{tab:Table1}
	\end{table}
	
	\subsection{Explicit examples with two assets}
	
	Consider the following artificial example where we want to compute bounds on the price of a basket option on a basket with two assets whose prices are given by $x_1$ and $x_2$, respectively. As a payoff function we choose $\max(0, 1/2x_1 + 1/2x_2-K)$. We assume we can observe the prices of two single call options on each asset. The corresponding optimization program is given in \eqref{twoassets} below.

	\begin{equation}\label{twoassets}
	\begin{aligned}
	\sup_{\mu \in \mathcal{M}(\mathbb{R}^2_+)_+} / \inf_{\mu \in \mathcal{M}(\mathbb{R}^2_+)_+}& \; \int_{\mathbb{R}^2_+}\max\left(0, \frac{1}{2}x_1+\frac{1}{2}x_2-K\right)\mathrm{d} \mu(\textbf{x}) \\
	\text{s.t. } & \; \int_{\mathbb{R}^2_+} \max(0, x_i-k_{x_i,j}) \mathrm{d}\mu(\textbf{x}) = a_{x_i,j} \,, \text{ for } i,j = 1,2\\
	& \; \int_{\mathbb{R}^2_+} \|\textbf{x}\|_2^2 \mathrm{d}\mu(\textbf{x}) \le M \\
	& \; \int_{\mathbb{R}^2_+}\mathrm{d}\mu(\textbf{x})=1.
	\end{aligned}
	\end{equation}
	To solve this numerically we slice up the domain into an irregular grid along the kinks of the $\max$-functions, under the assumption that the support of the optimal solution is contained in $[0,B]^2$ for some $B \in \mathbb{R}$. The domain then may look as depicted in Figure~\ref{fig:gridTiles}, where the dotted lines indicate where the objective ascends from 0, i.e., where $0.5x_1+0.5x_2-K = 0$. We index the tiles from bottom to top, left to right. 
	\begin{figure}
		\centering
		\begin{tikzpicture}[thick,scale=0.9, every node/.style={scale=0.8}]\label{Grid split}
		\draw[ ->] (1,1) -- (11.5,1); 
		\draw   (11,0.9) -- (11,11.5);
		\draw	(11.5,11) -- (0.9,11); 
		\draw[thick, ->] (1,1) -- (1,11.5); 
		\draw (4,0.9) -- (4,11.5); 
		\draw (9,0.9) -- (9,11.5); 
		\draw (0.9,5) -- (11.5,5); 
		\draw (0.9,7.5) -- (11.5,7.5); 
		\draw[dotted] (1,10.5) -- (4,9);
		\draw[dotted] (4,9) -- (9,4);
		\draw[dotted] (10.5,1) -- (9,4);
		\node at (11,0.7) {$B$};
		\node at (0.7,11) {$B$};
		\node at (0.7,10.5) {$2K$};
		\node at (10.5,0.7) {$2K$};
		\node at (0.4,7.5) {$k_{x_2,2}$};
		\node at (0.4,5) {$k_{x_2,1}$};
		\node at (9,0.6) {$k_{x_1,2}$};
		\node at (4,0.6) {$k_{x_1,1}$};
		\node at (2,2 ) {$1$};
		\node at (2,6 ) {$2$};
		\node at (2.5,8.5 ) {$3$};
		\node at (3,10 ) {$4$};
		\node at (7,2 ) {$5$};
		\node at (8.75,4.6 ) {$6$};
		\node at (6,6 ) {$7$};
		\node at (8,6.5 ) {$8$};
		\node at (4.5,8 ) {$9$};
		\node at (7,10 ) {$10$};
		\node at (9.5,1.75 ) {$11$};
		\node at (10.5,3 ) {$12$};
		\node at (10,6 ) {$13$};
		\node at (10,9 ) {$14$};
		\end{tikzpicture}
		\caption{Example of how the support might be split}
		\label{fig:gridTiles}
	\end{figure}
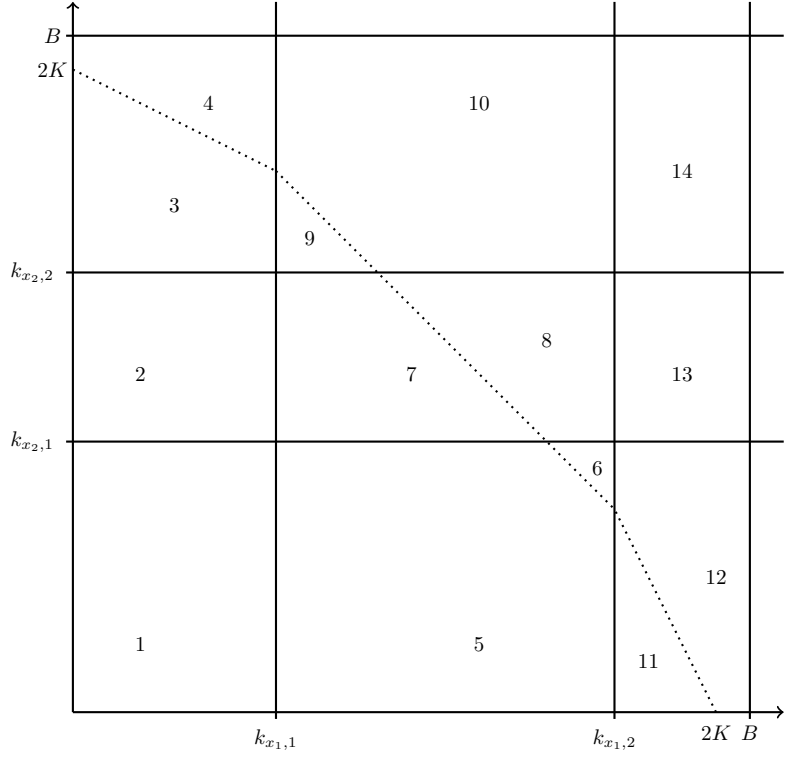
	For each tile $i$ in the grid we introduce a new measure $\mu_i$. For example for tile $12$ in Figure~\ref{fig:gridTiles} we get
	\begin{align*}
	\mathrm{supp}(\mu_{12}) = \{ \textbf{x} \in \mathbb{R}^2: &\  (B-x_1)(x_1-k_{x_1,2})\ge0\,,  \ x_2(k_{x_2,1}-x_2)\ge0 \,, \\
	& \ 1/2x_1+1/2x_2-K \ge0 \} \,.
	\end{align*}
	%
	Consider the following $(\text{strike}, \text{price})$ pairs
	\begin{itemize}
		\item $x_1$: $(100,12)$, $(110,3)$
		\item $x_2$: $(102,10)$, $(107,6)$
	\end{itemize}
	and let $M = 200\,000$, $B = 400$ and $K = 105$.
	
	Applying the above described procedure to problem~\eqref{twoassets} with the data given above results in problem~\eqref{numexp3}. Note that \eqref{twoassets} and \eqref{numexp3} are equivalent. With respect to Figure~\ref{fig:gridTiles} the index sets $J_i$ for $i = 0, 1, \dots, 4$ correspond to the sets on which the functions which define problem~\eqref{twoassets} are not identically zero, i.e., 
	\begin{itemize}
		\item $\max(0, \frac{1}{2}x_1+\frac{1}{2}x_2-K) = \frac{1}{2}x_1+\frac{1}{2}x_2-K$ on $J_0 = \{4,6,8,10,12,13,14\}$
		\item $\max(0,x_1-k_{x_1,1}) = x_1-k_{x_1,1}$ on $J_1 = \{ 5, 6, \dots, 14\}$
		\item $\max(0,x_1-k_{x_1,2}) = x_1-k_{x_1,2}$ on $J_2 = \{ 11, \dots, 14\}$
		\item $\max(0,x_2-k_{x_2,1}) = x_2-k_{x_2,1}$ on $J_3 = \{ 2, 3, 4, 7, 8, 9, 10, 13, 14\}$
		\item $\max(0,x_2-k_{x_2,2}) = x_2-k_{x_2,2}$ on $J_4 = \{ 3, 4, 9, 10, 14\}$.
	\end{itemize}
	Thus, we obtain the following problem:
	\begin{equation} \label{numexp3}
	\begin{aligned}
	\sup_{\mu_i}/\inf_{\mu_i}  \; & \sum_{i \in J_0} \int \left(\frac{1}{2}x_1+\frac{1}{2}x_2-K\right) \,\mathrm{d}\mu_i(\textbf{x}) \\
	\text{ s.t. } &\sum_{j \in J_1 } \int (x_1-k_{x_1,1})\, \mathrm{d}\mu_j(\textbf{x}) = a_{x,1}   \\
	&\sum_{j \in J_2 } \int (x_1-k_{x_1,2})\, \mathrm{d}\mu_j(\textbf{x}) = a_{x_1,2} \\
	&\sum_{j \in J_3 } \int (x_2-k_{x_2,1}) \,\mathrm{d}\mu_j(\textbf{x}) = a_{x_2,1} \\
	&\sum_{j \in J_4 } \int (x_2-k_{x_2,2})\, \mathrm{d}\mu_j(\textbf{x}) = a_{x_2,2} \\
	& \sum_{i=1}^{14} \int\mathrm{d}\mu_i(\textbf{x}) = 1 \\
	& \sum_{i=1}^{14} \int x_1^2+x_2^2\,\mathrm{d}\mu_i(\textbf{x}) \le M\,. \\
	\end{aligned}
	\end{equation}
	Applying the Moment-SOS hierarchy to this problem and solving the first level results in an upper bound of $7.4$ and a lower bound of $2.387$, which are in the optimal values of \eqref{twoassets}. 
	The SDP consisted of $14 \times 15 = 210$ variables, $80$ LMIs involving matrices of size $3 \times 3$, $4$ equality constraints and $211$ inequality constraints.

	\subsubsection*{Varying strikes}
	We are now going to give an example to see how changing the strike price affects the optimal values of the optimization problems. Consider the data presented in Table~\ref{tab:Table2} and let the objective function be $\max(0, 1/2x_1+1/2x_2-K)$, $B = 400$ and 
	$M=200\,000$. The optimal values are given in Table~\ref{tab:Table4}. All values stem from the first level of the Moment-SOS hierarchy and increasing the level up to level $10$ did not change the objective values. 
	For each of the strike prices specified in Table \ref{tab:Table2} the resulting programm for the first level of the hierarchy consisted of $47 \times 15 = 705$ variables, $257$ LMIs, $11$ equality constraints and $706$ inequality constraints. All moment and localizing matrices are of size $3 \times 3$.
	
	\begin{table}
		\centering
		
		\begin{tabular}{|c||c|c|c|c|c|}
			\hline
			$i$ & $1$ & $2$ & $3$ & $4$ & $5$ \\
			\hline
			$k_{x_1,i}$ & $90$ &  $95$ & $100$ & $110$ & $120$\\
			\hline
			$a_{x_1,i}$ & $20$ & $15.5$ & $12$ & $5.5$ & $1$\\
			\hline
			$k_{x_2,i}$ & $90$ & $96$ & $102$ & $107$ & $115$\\
			\hline
			$a_{x_2,i}$ & $20.5$ & $15$ & $10$ & $6$ & $0.75$\\
			\hline
		\end{tabular}
		\caption{Strikes and corresponding prices for European call options}
		\label{tab:Table2}
	\end{table}

	\begin{table}
		\centering
		\begin{tabular}{|c|c|c|c|c|c|c|}
			\hline
			$K$ & $90$ & $95$ & $100$ & $105$ & $110$ & $115$ \\
			\hline
			\hline
			lower bound on price & $16.875$ & $12.792$ & $8.708$ & $4.625$ & $1.675$ & $0.0$\\
			\hline
			computation time [s] & $0.12$ & $0.14$ & $0.13$ & $0.15$ & $0.14$ & $0.13$ \\
			\hline
			\hline
			upper bound on price & $20.25$ & $15.7$ & $11.55$ & $8.016$ & $4.75$ & $2$ \\
			\hline
			computation time [s] & $0.12$ & $0.14$ & $0.16$ & $0.15$ & $0.15$ & $0.14$ \\
			\hline
		\end{tabular}
		\caption{Optimal lower and upper bounds w.r.t.~the data given in Table~\ref{tab:Table2}}
		\label{tab:Table4}
	\end{table}

	
	
	
	\subsection*{Currency Basket}
	A currency basket is simply a way to determine the value of a national currency by calculating the weighted average of exchange rates of selected foreign currencies. These objects became popular in 1971 after the abolition of the gold standard. Options on currency baskets are attractive tools for multinational corporations to manage exposure to multiple currencies. Consider the following currency basket option on Euro and British Pounds in US Dollars. For both EUR/USD and GBP/USD two options are observable in the form $(strike, price)$:
	
	\begin{itemize}
		\item EUR/USD: $\{ (135.5,2.77), (138.5,1.17) \}$
		\item GBP/USD: $\{ (116,2.21), (119,0.67)\}$
	\end{itemize}
	
	We choose the weights $(2/3, 1/3)$ for the objective function, i.e. $\varphi(\textbf{x}) = \max(0, 2/3x_1+1/3x_2-K)$ and we  compute bounds for different values of $K$. We obtain an optimization problem similar to \eqref{twoassets}. The optimal values for the first level of the hierarchy are shown in Table~\ref{tab:TableCurrency}. 
	In the corresponding optimization problem the domain is partitioned into $14$ sets, for each of which $15$ moment variables are introduced. In total there are $14 \times 15 = 210$ variables, $80$ LMIs each involving a matrix of size $3 \times 3$, $211$ inequality constraints and $5$ equality constraints.
	For this particular example, it is clear that the bounds are not very useful in practice. This is, however, not due to our approach but to the number of data point given. In practice, there are more observable options available, improving the bounds that can be obtained. 
	
	\begin{table}
		\centering
		\begin{tabular}{|c|c|c|c|c|c|c|}
			\hline
			$K$ & $100$ & $105$ & $110$ & $115$ & $120$  \\
			\hline
			\hline
			lower bound on price & $1.4933$ & $1.2599$ & $1.0266$ & $0.7933$ & $0.56$  \\
			\hline
			computation time [s] & $0.22$ & $0.23$ & $0.22$ & $0.20$ & $0.16$  \\
			\hline
			\hline
			upper bound on price & $31.5834$ & $26.5833$ & $21.5833$ & $16.5833$ & $11.5833$ \\
			\hline
			computation time [s] & $0.15$ & $0.16$ & $0.18$ & $0.18$ & $0.16$  \\
			\hline
		\end{tabular}
		\caption{Optimal lower and upper bounds for a currency basket option with different strikes for level $r = 1$}
		\label{tab:TableCurrency}
	\end{table}
	
	\subsection*{Example from Boyle and Lin \cite{boyle}}
	
	In this example we compute bounds for a different type of option. We assume we only have data like mean, variance and correlation of the assets under the risk-neutral pricing measure available, instead of observable option prices with different strikes. The type of option is specified through the payoff function, which will be given by $\max(0, \max(x_1, \dots, x_n)-K)$ in this case. This type of option is called {\em call on max}. It is based on $n$ assets $S_1, \dots, S_n$, and gives the owner the right to buy the asset which at maturity is the most valuable for the predetermined strike $K$. 
	
	The data in the following example is taken from Boyle and Lin \cite{boyle}, where they introduced a different method to compute upper bounds. Consider three assets with means $(44.21,44.21,44.21)$ and the covariance matrix  given by 
	\[
	C = \begin{bmatrix}
	184.04 & 164.88 & 164.88 \\
	164.88 & 184.04 & 164.88 \\
	164.88 & 164.88 & 184.04
	\end{bmatrix}.
	\]
	Then, in our setting, the smallest upper bound on the price on the call on max option on these three assets is the optimal value of the following optimization problem:
	
	\begin{equation}\label{twoassetsCor}
	\begin{aligned}
	\sup_{\mu \in \mathcal{M}(\mathbb{R}^3_+)_+}  \; &\int_{\mathbb{R}^3_+}\max\left(0,\max(x_1,x_2,x_3)-K\right)\mathrm{d} \mu(\textbf{x}) \\
	\text{s.t. } & \; \int_{\mathbb{R}^3_+} x_i \mathrm{d}\mu(\textbf{x}) = 44.21\,, \text{for } i = 1,2,3\\
	& \; \int_{\mathbb{R}^3_+} (x_i-44.21)(x_j-44.21) \mathrm{d}\mu(\textbf{x}) = C_{i,j} \,, \text{for } i,j = 1,2,3 \\
	& \; \int_{\mathbb{R}^3_+} \|\textbf{x}\|_2^2 \mathrm{d}\mu(\textbf{x}) \le M \\
	& \; \int_{\mathbb{R}^3_+}\mathrm{d}\mu(\textbf{x})=1
	\end{aligned}
	\end{equation}
	The upper and lower bounds we obtain for different strikes \[
	K \in \{ 30, 35, 40 ,45, 50 \}
	\]
	are given in Table~\ref{tab:Table3} as well as the bounds obtained by Boyle and Lin. 
	Since all constraint functions are polynomial the only function contributing to the partition is the objective $\max\{0,\max(x_1,x_2,x_3)-K \}$. The resulting partition consists of 4 sets. To solve the first level of the hierarchy we introduce $4 \times 35 = 140$ moment variables. The final problem has $13$ equality constraints, $61$ inequality constraints and $22$ LMIs, each involving a matrix of size $4\times 4$.
	As in the previous example, the weakness of the bound is due to the fact that not enough information is available and is not inherent to the approach. Note that in their paper, Boyle and Lin only give a procedure for upper bounds. Also, in the original reference Boyle and Lin include a discount factor of $\exp(-0.1)$ to account for an assumed risk free interest rate. This has no effect on the optimization problem, they simply multiply their solution by the discount factor in the end.
	
	\begin{table}
		\centering
		\begin{tabular}{|c|c|c|c|c|c|}
			\hline
			$K$ & $30$ & $35$ & $40$ & $45$ & $50$ \\
			\hline
			Boyle \& Lin \cite{boyle} & $21.51$ & $17.17$ & $13.2$ & $9.84$ & $7.3$ \\
			\hline
			\hline
			upper bound on price & $21.51$ & $17.17$ & $13.2$ & $9.84$ & $7.3$\\
			\hline
			computation time [s]  & $0.02$ & $0.01$ & $0.01$ & $0.02$ & $0.02$\\
			\hline
			\hline
			lower bound on price & $14.21$ & $9.21$ & $4.21$ & $0$ & $0$\\
			\hline
			computation time [s]  & $0.02$ & $0.02$ & $0.01$ & $0.01$ & $0.01$\\
			\hline
		\end{tabular}
		\caption{Revisiting an example from Boyle and Lin, computing bounds on prices of a basket options given means and covariance of the underlying assets for different strikes.}
		\label{tab:Table3}
	\end{table}

	\subsection*{Basket option on tech stocks}
	As a last example we consider four different tech stocks, namely Apple Inc. (AAPL), Meta Platforms, Inc. (FB), Nvidia Corporation (NVDA), Qualcomm Incorporated (QCOM).
	Suppose one wants to price a basket option on these given the data provided in Table~\ref{tab:TableTechOptions} with payoff function $\max(0, \frac{1}{4}(x_1+\dots + x_4)-K)$, where the $x_i$ are the prices of the stocks of the given companies.
	The bounds obtained by solving the first level of the hierarchy for different strike prices are shown in Table~\ref{tab:TableResultsTech}. We set $B = 400$ and $M = 200\,000$.
	For this problem with $K = 140$, the partition consisted of $1938$ sets, for each of which we introduce $70$ variables, making $135\,660$ variables, $135\,661$ inequality constraints, $21$ equality constriants and $18\,726$ LMIs, each involving a $5\times 5$ matrix.   
	It is clear that the size of the partition necessary to compute these bound grows exponentially in the number of assets considered since it is lower bounded by $\prod_{i=1}^n N_i$. Even though for low levels of the hierarchy the involved matrices are very small, size of the partition is the limiting factor in the computations, since for every subset we need to introduce moment variables and at least $n$ LMIs. Note that changing $K$ may slightly change the number of partitions.
	\begin{table}
		\centering
		\begin{tabular}{|c||c|c|c|c|c|}
			\hline
			Company & (strike, price) & (strike, price)& (strike, price) & (strike, price) & (strike, price) \\
			\hline
			AAPL & $(120,45.2)$ &  $(130,35.7)$ & $(145,21.75)$ & $(160,9.1)$ & $(170,3.35)$\\
			\hline
			FB & $(155,52.7)$ & $(170,38.5)$ & $(180,29.85)$ & $(190,22)$ & $(200,14.75)$\\
			\hline
			NVDA & $(175,57.9)$ & $(180,53.2)$ & $(190,43.85)$ & $(195,39.35)$ & $(227.5,10.75)$\\
			\hline
			QCOM & $(130,35.35)$ & $(145,20.5)$ & $(157.5,8.8)$ & $(167.5,2.32)$ & $(175,0.47)$\\
			\hline
		\end{tabular}
		\caption{Strikes and corresponding prices for European call options observed on March 1st 2022, all prices in USD. }
		\label{tab:TableTechOptions}
	\end{table}

	\begin{table}
		\centering
		\begin{tabular}{|c|c|c|c|c|c|c|c|}
			\hline
			$K$ & $140$ & $150$ & $160$ & $170$ & $180$ & $190$ & $200$ \\
			\hline
			upper bound on price & $52.79$ & $42.89$ & $33.48$ & $24.53$ & $15.68$ & $8.51$ & $6.99$\\
			\hline
			computation time [s]  & $32.38$ & $43.10$ & $40.92$ & $42.89$ & $44.11$ & $46.46$ & $39.44$\\
			\hline 
			\hline
			lower bound on price & $46.26$ & $36.26$ & $26.27$ & $16.28$ & $6.28$ & $0.0$ & $0.0$\\
			\hline
			computation time [s]  & $47.16$ & $50.48$ & $53.89$ & $52.67$ & $57.01$ & $45.74$ & $37.70$\\
			\hline
		\end{tabular}
		\caption{Bounds for basket options on tech firms subject to observable data of Table~\ref{tab:TableTechOptions}}
		\label{tab:TableResultsTech}
	\end{table}
\subsection{Lasserre hierarchy of inner range}\label{sec:innerrange}
The one considered in this section, known as the Lasserre measure-based hierarchy of inner bounds introduced by Lasserre \cite{lasserre2010}, consists of fixing a reference measure $\nu$ on $\mathbb{R}^n_+$ such that $\nu(\mathbb{R}^n_+)< \infty$ and then approximating the density function of the optimal measure $\mu$ for \eqref{prob1} by SOS polynomials $h_r(\textbf{x}) \in \Sigma[\textbf{x}]_r$, such that $\mathrm{d}\mu(\textbf{x}) = h_r(\textbf{x})\mathrm{d}\nu(\textbf{x})$. This has the advantage that instead of searching for an optimal measure in the infinite dimensional cone $\mathcal{M}(\mathbb{R}^n_+)_+$ we optimize over the set of sums of squares of fixed degree, which can be done with SDP techniques. Opposed to before, the cone of measures $\mathcal{M}(\mathbb{R}^n_+)_+$ is here approximated from inside, while before, we used an outer approximation.  A possible choice for the reference measure is

\[ 
\mathrm{d}\nu(\textbf{x}) = \exp\left(- \sum_{i=1}^{n} x_i \right)\mathrm{d}\textbf{x}\,.
\]

An important assumption on the reference measure is that its moments must be available in closed form or efficiently computable. In the case above the moments are given by $\int_{\mathbb{R}^n_+} \textbf{x}^\alpha d\nu(\textbf{x}) = \alpha!$. The level $r$ relaxation of problem \eqref{prob1} can be formulated as follows
\begin{equation}\label{measureupper}
	\begin{aligned}
		\inf_{h_r \in \Sigma[\textbf{x}]_r} & \int_{\mathbb{R}^n_+} \varphi(\textbf{x})h_r(\textbf{x})\mathrm{d}\nu(\textbf{x}) \\
		\text{ s.t. } & \int_{\mathbb{R}^n_+}f_{i,j}(\textbf{x})h_r(\textbf{x})\mathrm{d}\nu(\textbf{x}) = q_{i,j}\,, \text{ for } i \in [n], j \in[N_i] \\
		& \int_{\mathbb{R}^n_+} f_\ell(\textbf{x})h_r(\textbf{x}) \mathrm{d}\nu(\textbf{x}) = p_\ell\,, \text{ for } \ell \in [m] \\
		& \int_{\mathbb{R}^n_+} \|\textbf{x}\|_2^2 h_r(\textbf{x}) \mathrm{d}\nu(\textbf{x}) \le M\,. \\
	\end{aligned}
\end{equation}
This problem can be cast as an SDP. It should be noted that the above SDP might be infeasible even if the GMP has an optimal solution. As a simple example consider the following constraint for some $\alpha \in \mathbb{N}^n$

\[
\int_{\mathbb{R}^n_+} \textbf{x}^\alpha \mathrm{d}\mu(\textbf{x}) = 0\,.
\]
While the atomic Dirac delta measure $\delta_0$ at $0$ certainly satisfies this equation, there is no $r \in \mathbb{N}$ such that there is a degree $r$ sos polynomial density function that does. One can, however, relax the constraints slightly, by searching for an $h_r$ such that one lands in (increasingly) close proximity of the right hand side. Consider the following generalized moment problem 

\begin{equation}\label{gmpdistrob}
	b_0 = \inf_{\nu \in \mathcal{P}(K)_+}  \left \{ \int_{K_0} f_0(\textbf{x})\mathrm{d}\nu(\textbf{x}) : \int_{K_i} f_i(\textbf{x})\mathrm{d}\nu(\textbf{x})=b_i \right\}\,, 
\end{equation}
where $ \mathcal{P}(K)_+$ is the set of probability measures on $K \subset \mathbb{R}^n$, $\mathrm{int}K \neq \emptyset$ and $K_i \subset K$ is closed for every $i = 0,1, \dots, m$. De Klerk et al. proved the following result in \cite{distrob}.

\begin{theorem}\label{thmdeklerk}
	Let $\mu$ be a reference measure with known (or efficiently computable) moments such that the moments are finite and $\int_K x_i^{2k}\mathrm{d}\mu(\textbf{x}) \le (2k)!M$ for some $M > 0$ and all $i \in [n], k \in \mathbb{N}$. If all $f_i$ for $i = 0,1, \dots, m$ are polynomials, then, as $ r \rightarrow \infty$ we have
	
	\[ \varepsilon(r): = \inf_{h \in \Sigma[\textbf{x}]_r} \max_{i = 0,1,\dots, m} \left \vert \int_{K_i} f_i(\textbf{x}) h(\textbf{x})\mathrm{d}\mu(\textbf{x}) - b_i \right \vert \]
	tends to zero ($\varepsilon(r) = o(1)$). 
\end{theorem}

This means that if we fix an $\varepsilon > 0$ and relax the equality constraints to an $\varepsilon$ neighborhood of the RHS, then we will eventually (for $r$ large enough) find a feasible solution for the relaxation such that the optimal value is at most $\varepsilon$ away from the true optimum. Theorem \ref{thmdeklerk} promises convergence but we cannot say anything about the rate at which $\varepsilon$ goes to zero. It shall be mentioned that adding the $\varepsilon(r)$ in the relaxation does not necessarily result in the inner range of the bounds of the sought option prices, since this is basically an outer approximation of the inner range. Another way to think of it is first relaxing the equality constraints of problem \eqref{measureupper} resulting in an increase of the possible range and then applying the inner approximation to the obtained optimization problem.  When adding the $\varepsilon_r$-relaxation it is clear that we cannot expect monotonicity of the bounds, which will become apparent in the numerical results of section \ref{sec:innerEx}.

\subsection{Univariate example}\label{sec:innerEx}

Consider the following example with data taken from \cite{bertsimas}. 

\begin{equation}\label{measureupperEx}
	\begin{aligned}
		\sup_{h_r \in \Sigma[x]_r} / \inf_{h_r \in \Sigma[x]_r} & \int_{\mathbb{R}_+} \max(0, x-105)h_r(x)\mathrm{d}\nu(x) \\
		\text{ s.t. } & \int_{\mathbb{R}_+}\max(0,x-100)h_r(x)\mathrm{d}\nu(x) = 8.375 \\
		& \int_{\mathbb{R}_+} \max(0,x-110)h_r(x) \mathrm{d}\nu(x) = 1.875 \\
		& \int_{\mathbb{R}_+}h_r(x)  \mathrm{d}\nu(x) = 1
	\end{aligned}
\end{equation}
We know that the optimal lower and upper bounds for this data set are $3.375$ and $5.125$, respectively. 
To improve the numerical stability of SDP \eqref{measureupperEx}, one can use a basis which is orthogonal on $\mathbb{R}_+$ w.r.t.~the measure $\mathrm{d}\nu(x) = \exp(-x)\mathrm{d}x$, namely the Laguerre basis defined by
\[
L_n(x) = \sum_{i = 0}^n \binom{n}{i}\frac{(-1)^i}{i!}x^i\,.
\]
These polynomials form an orthogonal system for the Hilbert space $L^2(\mathbb{R}_+, w(x)\mathrm{d}x)$ with $w(x) = \exp(-x)$, i.e., 
\[
\int_0^\infty L_n(x)L_m(x)\exp(-x)\mathrm{d}x = \begin{cases} 1, \text{ if } m = n \\
0, \text{ otherwise. }\end{cases}
\]
To implement the program we used the fact that 

\begin{equation}
	\int_{k}^\infty x^n \mathrm{d}\nu(x) = \exp(-k)\left(\sum_{\ell = 0}^{n}\frac{n!}{\ell !}k^\ell \right)
\end{equation}
and relaxed it to

\begin{equation}\label{measureupperEx2}
	\begin{aligned}
		\sup_{h_r \in \Sigma[x]_r} / \inf_{h_r \in \Sigma[x]_r} & \int_{\mathbb{R}_+} \max\left(0, x-\frac{105}{110}\right)h_r(x)\mathrm{d}\nu(x) \\
		\text{ s.t. } & \left| \int_{\mathbb{R}_+}\max\left(0,x-\frac{100}{110}\right)h_r(x)\mathrm{d}\nu(x) -\frac{8.375}{110} \right| \le \varepsilon_r \\
		& \left| \int_{\mathbb{R}_+} \max\left(0,x-\frac{110}{110}\right)h_r(x) \mathrm{d}\nu(x) - \frac{1.875}{110} \right| \le \varepsilon_r \\
		&\left| \int_{\mathbb{R}_+}h_r(x)  \mathrm{d}\nu(x) - \frac{1}{110} \right| \le \varepsilon_r\,.
	\end{aligned}
\end{equation}
As a normalization step, we divided the data by $110$. We indicate in Table \ref{tab:table7inner} how the optimal values change if for level $r$ we choose $\varepsilon_r$ to be the smallest value such that the corresponding relaxation still has a feasible solution. In other words, decreasing $\varepsilon_r$ in this cases results in infeasibility. Observe that no monotonicity appears, which is expected because the equality constraint is relaxed. We mention that in Table \ref{tab:table7inner} for $r \in \{6,7\}$ MOSEK could not solve the maximization problem. However, the upper bound approximations were already reasonably accurate at the previous levels. 
It seems that the approach considered in section \ref{sec:outer} is superior to the one presented in this section, since there we get the optimal values of $5.125$ and $3.375$ for the first level of the hierarchy already. Especially, when considering the fact that increasing $r$ quickly results in numerical problems and the problem is highly susceptible to small changes in $\varepsilon_r$. 
Additionally, it is difficult to estimate how much the $\varepsilon$ relaxation perturbs the optimal value of the optimization problem. 

\begin{table}
	\centering
	\begin{tabular}{|c|c|c|c|c|c|c|c|}
		\hline
		$r$ & $2$ & $3$ & $4$ & $5$  & $6$ & $7$ & $\infty$\\
		\hline
		$\varepsilon_r$ & $0.0273$ & $0.02525$ & $0.022125$ & $0.01755$ & $0.0161$ & $0.0161$ & $0$\\
		\hline
		\hline
		upper bound & $5.1279$ & $5.1366$ & $5.1288$ & 
		$5.1264$ & - & - & $5.125$\\
		\hline
		time in s & $0.01$ & $0.01$ & $0.01$ & $0.01$ & - & -& - \\
		\hline 
		\hline
		lower bound & $5.122$ & $5.1136$ & $5.1221$ & $5.1251$ & $4.224$ & $3.3522$ & $3.375$\\
		\hline
		time in s & $0.01$ & $0.01$ & $0.01$ & $0.01$ & $0.02$ & $0.03$ & -\\
		\hline
	\end{tabular}
	\caption{Optimal solutions for the level-$r$ relaxation of the measure-based Lasserre hierarchy applied to the $\varepsilon_r$ relaxation given in \eqref{measureupperEx2} for Laguerre basis with varying $\varepsilon_r$ for $r = 2,\dots,7$. The $\varepsilon_r$ are the smallest possible such that the resulting SDP still has a feasible solution.}
	\label{tab:table7inner}
\end{table}

\section{Conclusion and further research}

In this section we reflect on our results,  and state open questions that could be further studied. The model we considered has the advantage that it combines different possibilities of using observable data. Option prices with different strikes as well moment information like mean, (co-)variance etc, can be taken into account, which is very useful in practice.  The Moment-SOS hierarchy, which was used to obtain the outer range delivers good approximations for low hierarchy levels. The method for the inner range quickly fails but in the considered cases still gave reasonable bounds. However, comparing the two, the outer range clearly outperformed the inner range.  

Regarding the compactness argument it should be noted that in practice it might be prohibitive to carry out the core variety procedure in a setting with many assets and constraints. In a setting where it becomes too difficult one can of course start with an educated guess for the $B$ defined in section \ref{sec:boundsupport} and compute bounds for this $B$ and a larger one $\hat{B} > B$, and increase $B$ until the optimal values no longer change. 

\subsection*{Acknowledgements}
This  research  was  supported  by  the  European  Union’s  Horizon  2020  research  and  innovation  programme under the Marie Sk\l odowska-Curie grant agreement N. 813211 (POEMA). The second author would like to thank Corbinian Schlosser for insightful discussions.


\end{document}